\documentclass{amsart}
\usepackage{graphicx}
\usepackage{amsmath,amssymb,array}
\usepackage[T1]{fontenc}
\usepackage{amsfonts}
\usepackage{mathabx}
\usepackage{amsmath}
\usepackage{amssymb}
\usepackage{amsthm}
\usepackage{graphicx}
\usepackage{enumitem}
\usepackage{xcolor}
\usepackage[utf8]{inputenc}
\usepackage{hyperref}
\usepackage{mathrsfs}
\usepackage[toc,page]{appendix}
\usepackage{fancyhdr}
\usepackage{tikz}
\usepackage{tikz-cd}
\usepackage{longtable}
\usepackage{geometry}
\usepackage{textcomp}
\usetikzlibrary{trees}
\usetikzlibrary{arrows}
\usepackage{multirow}
\usepackage[all,cmtip]{xy}
\usepackage{youngtab}
\usepackage{mathdots}

\usepackage{hyperref}
\usepackage{amssymb}
\usepackage{mathdots}
\usepackage{dynkin-diagrams}
\usepackage{booktabs}
\usepackage{geometry}

\usepackage{mathrsfs}
\usepackage{todonotes}
\usepackage{longtable}

\definecolor{yqyqyq}{rgb}{0.5019607843137255,0.5019607843137255,0.5019607843137255}\definecolor{uuuuuu}{rgb}{0.26666666666666666,0.26666666666666666,0.26666666666666666}
\definecolor{uququq}{rgb}{0.25098039215686274,0.25098039215686274,0.25098039215686274}
\definecolor{wwwwww}{rgb}{0.4,0.4,0.4}
\definecolor{uuuuuu}{rgb}{0.26666666666666666,0.26666666666666666,0.26666666666666666}

\setlist[itemize]{leftmargin=6mm}

\newcommand{\Pj}{\mathbb{P}}

\newcommand{\Q}{\mathbb Q}

\newcommand{\Aut}{\operatorname{Aut}}

\DeclareMathOperator{\Cl}{Cl}

\DeclareMathOperator{\rk}{rk}

\DeclareMathOperator{\Hom}{Hom}

\DeclareMathOperator{\Exc}{Exc}

\DeclareMathOperator{\Nef}{Nef}
\DeclareMathOperator{\Mov}{Mov}
\DeclareMathOperator{\Pic}{Pic}

\newtheorem{thm}{Theorem}[section]

\newtheorem{Lemma}[thm]{Lemma}
\newtheorem{Proposition}[thm]{Proposition}

\newtheorem{Corollary}[thm]{Corollary}

\theoremstyle{definition}

\newtheorem{Definition}[thm]{Definition}
\newtheorem{Remark}[thm]{Remark}

\newtheorem{Example}[thm]{Example}

\newtheorem{say}[thm]{}

\newtheorem{theorem}{Theorem}[section]

\theoremstyle{definition}

\theoremstyle{remark}
\newtheorem{rmkk}[thm]{Remark}
\newtheorem{exe}[thm]{Example}

\newcommand{\qee}{\mbox{\hspace{0.2mm}}\hfill$\triangle$}



\title{Mori dream singular $K3$ surfaces}

\author{Antonio Laface}
\address{Antonio Laface, Departamento de Matematica, Universidad de Concepci\'on, Casilla 160-C, Concepci\'on, Chile}
\email{alaface@udec.cl}

\author{Alex Massarenti}
\address{Alex Massarenti, Dipartimento di Matematica e Informatica, Università di Ferrara, Via Machiavelli 30, 44121 Ferrara, Italy}
\email{msslxa@unife.it}

\author{William D. Montoya}
\address{William D. Montoya, Instituto de Matemática, Estatística e Computação Científica, Universidade Estadual de Campinas (UNICAMP), Rua Sérgio Buarque de Holanda 651, 13083-859, Campinas, SP, Brazil}
\email{wmontoya@ime.unicamp.br}

\begin{document}

\begin{abstract}
We take a first step towards the classification of singular Mori dream $K3$ surfaces. We prove that if the Picard lattice of a singular $K3$ surface is Mori dream, then the surface is Mori dream. Moreover, we show that for singular $K3$ surfaces, of Picard rank two, being Mori dream is equivalent to contain two negative curves intersecting each other, and apply this result to study Mori dreamness of $K3$ surfaces with a singular point of type $A_n$.   
\end{abstract}

\maketitle

\setcounter{tocdepth}{1}
\tableofcontents

\section{Introduction}
In this paper we begin the study of a possible classification of $K3$ surfaces that are Mori dream spaces. This last notion has been introduced by Hu and Keel in \cite{HuKeel}, and is motivated by the fact that these spaces are well behaved from the point of view of Mori's minimal model program. 

Mori dream smooth $K3$ surfaces have been classified in terms of Picard lattices \cite{ahl}. We recall that the Picard group of a $K3$ surface $X$ embeds, as a sublattice $\Lambda_X$, of the so called $K3$ lattice $\Lambda_{K3}$. We will say that a sublattice of $\Lambda_{K3}$ is a Mori dream lattice if it is the sublattice associated to a smooth Mori dream $K3$ surface. Furthermore, a smooth $K3$ surface of Picard rank two is Mori dream if and only if it contains a divisor  $D$  of self-intersection  $D^2 \in \{-2,0\}$. For higher Picard rank, there exists a comprehensive list of $178$ lattices corresponding to Mori dream $K3$ surfaces
\cite[Theorem 5.1.5.3]{adhl}.

In the smooth setting we prove, in Theorem \ref{moridreamthm}, that under mild hypotheses for smooth $K3$ surfaces of Picard rank two being Mori dream depends only on the discriminant of the intersection form:

\begin{thm}\label{main1}
Let $X,Y$ be two smooth $K3$ surfaces of Picard rank two such that the discriminant $d_X,d_Y$ of their intersection forms are equal. Assume that both $X$ and $Y$ contain a class of self-intersection $2a$ with $0\leq 2a\leq 6$. Then $X$ is Mori dream if and only if $Y$ is Mori dream. 
\end{thm} 

We apply Theorem \ref{main1} to identify which smooth quartic $K3$ 
surfaces in the Noether-Lefschetz locus, defined by the 
determinantal components, are Mori dream spaces, as detailed in 
Corollary \ref{cor:det}.

However, as we will show, the singular case is much more intricate and difficult to deal with than the smooth one. For instance, as shown in Example \ref{SpecF}, there are families of non Mori dream smooth $K3$ surfaces specializing to Mori dream singular $K3$ surfaces. In this example, the specialization is constructed as a contraction of a smooth non Mori dream $K3$ surface. More generally, determining which contractions of $K3$ surfaces are Mori dream remains an open problem.
Our first result in Theorem \ref{mdsk3}, which generalizes \cite[Proposition 2.10]{gar}, is the following:

\begin{thm}\label{MDL}
Let $X$ be a singular $K3$ surface whose Picard lattice $\Lambda_X$ is Mori dream. Then $X$ is a Mori dream surface.
\end{thm}

However, Theorem \ref{MDL} leaves unsolved the question of identifying singular $K3$ surfaces that are Mori dream spaces, even though their Picard lattice is not Mori dream. We provide an answer in the case where the $K3$ surface has Picard rank two and at most one singularity of type $A_n$. Our main results, in Proposition \ref{pro:mds} and Theorem \ref{MainAn}, can be summarized as follows:

\begin{thm}\label{main2}
Let $X$ be a, possibly singular, projective $K3$ surface of Picard rank two. The following are equivalent:
\begin{itemize}
\item[\rm{(i)}] $X$ is a Mori dream space;
\item[\rm{(ii)}] there are two effective divisors $D_1,D_2$ on $X$ with $D_i^2\leq 0$ and $D_1\cdot D_2>0$.
\end{itemize}
Furthermore, assume that the singular locus of $X$ consists of a point of type $A_n$ and $X$ contains an irreducible curve of negative self-intersection. If $n\notin\{11,14,15\}$ then $X$ is Mori dream. 
\end{thm}

Most of the arguments used to prove Theorem \ref{main2} consist of an arithmetic part, using the theory of binary quadratic forms, to prove the existence of certain negative divisors, and a geometric part needed to prove their effectiveness. 

Finally, in Section \ref{c1s1} we apply Theorem \ref{main2} to provide explicit examples of singular Mori dream $K3$ surfaces embedded in $3$-dimensional weighted projective spaces.

\medskip

\subsection*{Acknowledgments} 
The first named author has been partially supported by Proyecto FONDECYT Regular n.~1230287. 

The second named author has been supported by the PRIN 2022 project-20223B5S8L "Birational geometry of moduli spaces and special varieties", and is a member of the Gruppo Nazionale per le Strutture Algebriche, Geometriche e le loro Applicazioni of the Istituto Nazionale di Alta Matematica "F. Severi" (GNSAGA-INDAM).

The third named author has been supported by the FAPESP postdoctoral grant 2019/23499-7 and BEPE Research Internship Abroad grant 2023/01360-2, and thanks for the hospitality the University of Ferrara and the University of Concepci\'on where part of this work was developed.

\section{Mori dream smooth $K3$ surfaces}
We recall the main definitions and prove that, under mild assumptions, a $K3$ surface $X$ of Picard rank two is Mori dream if and only if any $K3$ surface with the same discriminant as $X$ is Mori dream. 

\begin{Definition}
A $K3$ surface is a smooth projective surface $X$ with trivial canonical divisor: $K_X\sim 0$, and irregularity zero: $h^1(X,\mathcal{O}_X) = 0$. 
\end{Definition}

\begin{say}\label{K3lat}
If $X$ is a $K3$ surface then $H^2(X,\mathbb Z)$ is isometric to the $K3$ lattice
$$
\Lambda_{K3} = E_8^2\oplus U^3,
$$
which is an even unimodular lattice of signature $(3,19)$. The Picard group $L_X := \Pic(X)$ of $X$ embeds isometrically into $\Lambda_{K3}$ and its orthogonal complement is the transcendental lattice $T_X$ of $X$.

It follows that $L_X\oplus T_X$ has finite index in $\Lambda_{K3}$. This index equals the order of $L_X^*/L_X\simeq T_X^*/T_X$, where $L^* := \Hom(L,\mathbb Z)$ is the dual lattice of $L$ equipped with the rational quadratic form induced by $L$, we refer to \cite{hu} for further details.
\end{say}

\begin{Definition}\label{def:MDS} 
A normal projective $\mathbb{Q}$-factorial variety $X$ is called a \emph{Mori dream space}
if the following conditions hold:
\begin{itemize}
\item[-] $\Pic(X)$ is finitely generated, or equivalently $h^1(X,\mathcal{O}_X)=0$,
\item[-] $\Nef(X)$ is generated by the classes of finitely many semi-ample divisors,
\item[-] there is a finite collection of small $\mathbb{Q}$-factorial modifications
 $f_i: X \dasharrow X_i$, such that each $X_i$ satisfies the second condition above, and $
 \Mov(X) \ = \ \bigcup_i \  f_i^*(\Nef(X_i))$.
\end{itemize}
\end{Definition}

Let us recall~\cite[Theorem 5.1.5.1]{adhl}.

\begin{thm}\label{thm:mds1}
Let $X$ be a $K3$ surface. Then the following are equivalent:
\begin{enumerate}
    \item[\rm{(i)}] $X$ is a Mori dream surface.
    \item[\rm{(ii)}] The effective cone ${\rm Eff}(X)\subset {\rm Pic}_{\Q}(X)$ is polyhedral.
    \item[\rm{(iii)}] The automorphism group of $X$ is finite.
\end{enumerate}
\end{thm}

Moreover, if the Picard number of $X$ is at least three, then $X$ is a Mori dream surface if and only if it contains only finitely many smooth rational curves. In this case, these curves are $(-2)$-curves and their classes generate the effective cone. In \cite{adhl} Mori dream $K3$ surfaces are characterized using the classification of $K3$ surfaces with finite automorphism group. 

The following result is a consequence of~\cite[Theorem 1.12.4]{Ni} and of the global Torelli theorem.

\begin{thm}\label{thm-ni}
For any even lattice $\Lambda$ of signature $(1,\rho)$ with $\rho\leq 10$, there exists a $K3$ surface $X$ such that $\Lambda_X\simeq \Lambda$.
\end{thm}

From now on we will assume $X$ to be a $K3$ surface of Picard rank two. In this case an improvement of Theorem~\ref{thm:mds1} is given by \cite[Theorem 5.1.5.3]{adhl} which we now recall.

\begin{thm}\label{thethm}
Let $X$ be a $K3$ surface of Picard number two. Then the following are equivalent:
\begin{itemize}
\item[\rm{(i)}] $X$ is a Mori dream surface;
\item[\rm{(ii)}] $\Pic (X)$ contains a class $D$ with $D^2\in\{-2,0\}$.
\end{itemize}
\end{thm}

\begin{Definition}
Let $X$ be a $K3$ surface with Picard lattice $\Lambda_X$. The \textit{discriminant} $d_X$ of $X$ (or of $\Lambda_X$) is the determinant of the intersection matrix of any basis of $\Lambda_X$ multiplied by $-1$. When $\Lambda_X$ has rank two the intersection matrix and the discriminant are of the following form: 
\stepcounter{thm}
\begin{equation}\label{intmat}
   \left(\begin{array}{cc}
   2a  & b  \\
    b & 2c
\end{array}\right),
\quad
d_X = b^2-4ac.
\end{equation}   
\end{Definition}

\begin{Remark}\label{square}
There is a divisor class $D$ such that the condition $D^2=0$ holds exactly when the discriminant $d_X$ is a square.
The surface admits a divisor class $D$ such that $D^2=-2$ when the equation
$$
ax^2+bxy+cy^2=-1
$$
admits an integral solution. Moreover, after possibly applying transformations of the form $(x,y)\mapsto (x,x\pm y)$, we can assume that $-a\leq b \leq a$. Similarly, after possibly applying transformations of the form $(x,y)\mapsto (x,-y)$, we can assume that $0\leq b\leq a$.
\end{Remark}

Furthermore, we recall the following criterion when $X$ is a smooth quartic surface in $\mathbb P^3$.

\begin{Proposition}
Let $X$ be a smooth $K3$ surface of Picard rank two and intersection matrix
given by~\eqref{intmat}. Then the following are equivalent:
\begin{itemize}
\item[\rm{(i)}] ${\rm Pic}(X)$ contains a class $D$  with $D^2 = -2$; 
\item[\rm{(ii)}] the following equivalence of quadratic forms holds:
$$
 ax^2 + bxy + cy^2
 \equiv
 \begin{cases}
 -x^2 + \frac{1}{4}d_X y^2 & \text{if } b \text{ is even};\\
 -x^2 + xy + \frac{1}{4}(d_X-1) y^2 & \text{if }b \text{ is odd}.
 \end{cases}
$$
\end{itemize}
\end{Proposition}
\begin{proof}
First of all observe that the equation $D^2 = -2$ is equivalent to $ax^2 + bxy + cy^2 = -1$.
If the form is equivalent to one of the two displayed forms, then, by evaluating at $(1,0)$, it follows
that the equation $ax^2 + bxy + cy^2=-1$ admits an integral solution.

Now, assume that the equation $ax^2 + bxy + cy^2 = -1$ has an integral solution $(x_0,y_0)$. Then $\gcd(x_0,y_0) = 1$.
Acting with $GL(2,\mathbb Z)$ we can map $(x_0,y_0)$ to $(1,0)$. Then $2x^2 + bxy + cy^2$ is mapped
to $-x^2 + b'xy + c'y^2$. Applying sufficiently many linear maps of the form $(x,y) \mapsto (x,x\pm y)$
we obtain a form with the same coefficient of $x^2$ as the original form, and with the coefficient of $xy$ in $\{0,1\}$. Finally, the statement follows.
\end{proof}

\begin{thm}\label{moridreamthm}
Let $X,Y$ be two smooth $K3$ surfaces of Picard rank two such that $d_X = d_Y$. Assume that both $X$ and $Y$ contain
a class of self-intersection $2a$ with $0\leq 2a\leq 6$. Then $X$ is Mori dream if and only if $Y$ is Mori dream. 
\end{thm} 
\begin{proof}
If the discriminant is a square the claim follows by Remark~\ref{square}. So assume this is not the case. The intersection matrices of $X$ and $Y$
are 
$$   
 \left(\begin{array}{cc}
   2a  & b_X  \\
    b_X & 2c_X
\end{array}\right)
\text{ and }
 \left(\begin{array}{cc}
   2a  & b_Y  \\
    b_Y & 2c_Y
\end{array}\right)
$$
respectively. The corresponding quadratic forms are thus $ax^2+b_Xxy+c_Xy^2$ and $ax^2+b_Yxy+c_Yy^2$. By Remark~\ref{square} we can assume $0\leq b_X,b_Y\leq a$. On the other hand, the equality $d_X=d_Y$ 
yields $b_X^2\equiv b_Y^2\pmod {4a}$. When $a\leq 3$ all these conditions imply 
$b_X=b_Y$, so that the two intersection forms are equivalent.
\end{proof}

\begin{Example}
The assumption $a\leq 3$ in Theorem~\ref{moridreamthm} is necessary. For instance, when $a=4$ the following
intersection matrices
$$ 
 \left(\begin{array}{cc}
    8 & 4  \\
    4 & -2
\end{array}\right)
\text{ and }
 \left(\begin{array}{cc}
    8 & 0  \\
    0 & -4
\end{array}\right)
$$
define two $K3$ surfaces with Picard group of rank two and discriminant $32$. The first one is Mori dream, since it 
contains a class of self-intersection $-2$, while the second one is not Mori dream because
any class in $\Pic(X)$ has self-intersection which is a multiple of $4$.
\end{Example}

We present some families of smooth Mori dream $K3$ surfaces in $\Pj^3$ of Picard rank two. Recall that a degree $d$ surface $X = \{f = 0\}\subset \Pj^3$ is called determinantal if there exists a matrix $M$ whose entries are polynomials in the homogeneous coordinates of $\Pj^3$ and $f = \det(M)$. We recall the following result of \cite{LealLozanoMontserrat}.

\begin{thm} 
The smooth determinantal $K3$ surfaces in $\Pj^3$ are parametrized by five prime divisors 
$$\mathcal{D}_1,\dots\mathcal{D}_5\subset \mathbb{P}(H^0(\mathbb{P}^3,\mathcal{O}_{\mathbb{P}^3}(4)).$$ 
A very general surface in $\mathcal{D}_i$ has  Picard rank two and discriminant $20, 17, 16, 9, 12$ respectively.
\end{thm}

\begin{Corollary}
\label{cor:det}
A very general $K3$ surface in $\mathcal{D}_1$ is not a Mori dream space while a very general $K3$ surface in $\mathcal{D}_2, \mathcal{D}_3, \mathcal{D}_4 $ or $\mathcal{D}_5$ is a Mori dream space.
\end{Corollary}
\begin{proof} 
The discriminant of a very general surface $X$ in $\mathcal D_3$ or in $\mathcal D_4$ is a square, so that $X$ 
is Mori dream.

For the remaining divisors we are going to apply Theorem~\ref{moridreamthm}, which allows us to check Mori dreamness for any smooth quartic of Picard rank two with prescribed discriminant. The intersection matrices of a very general surface in $\mathcal D_1, \mathcal D_2$ or $\mathcal D_5$ are
\[
 \left(
  \begin{array}{cc}
   4&2\\
   2&-2
  \end{array}
 \right),
 \:
 \left(
  \begin{array}{cc}
   4&3\\
   3&-2
  \end{array}
 \right),
 \:
 \left(
  \begin{array}{cc}
   4&2\\
   2&-4
  \end{array}
 \right)
\]
respectively. For the first two note that the equation $D^2 = -2$ admits an integral solution, while for the third one there can not be a solution since $D^2$ is a multiple of $4$.
\end{proof}

\section{Mori dream singular $K3$ surfaces}
In this section we lay the foundations to study Mori dreamness of singular $K3$ surfaces.

\begin{Definition}
A singular $K3$ surface $X$ is a normal projective surface with at most rational double points, trivial canonical class: $K_X\sim 0$ and trivial irregularity: $h^1(X,\mathcal O_X) = 0$.
\end{Definition}

\begin{Lemma}\label{res}
Let $X$ be a singular $K3$ surface and $\pi: \widetilde{X}\to X$ its minimal resolution of singularities with exceptional locus supported on the irreducible curves $E_1,\dots,E_n$. Then the following hold:
\begin{itemize}
\item[\rm{(i)}] $\widetilde{X}$ is a smooth $K3$ surface.
\item[\rm{(ii)}] $\pi^*\Pic(X) = \langle E_1,\dots,E_n\rangle^\perp$.
\item[\rm{(iii)}] $\pi_*: \Cl(\widetilde{X})\to\Cl(X)$
induces an isomorphism of finite abelian groups
$$
\Cl(\widetilde{X})/\pi^*\Pic(X)\oplus\langle E_1,\dots,E_n\rangle\simeq\Cl(X)/\Pic(X).
$$
\end{itemize}
\end{Lemma}
\begin{proof}
To prove (i) it suffices to show that $K_{\widetilde{X}}\sim 0$ and $h^1(\widetilde{X},\mathcal O_{\widetilde{X}})=0$. Since $X$ has only rational double points the resolution is crepant, so that 
$K_X\sim 0$ implies $K_{\widetilde{X}}\sim 0$. The vanishing of $H^1(\widetilde{X},\mathcal O_{\widetilde{X}})$ is a consequence of $R^1\pi_*\mathcal O_X=0$, which holds since $X$ has rational singularities, and of $h^1(X,\mathcal O_X)=0$.

We prove (ii). The inclusion $\pi^*\Pic(X)\subseteq \langle E_1,\dots,E_n\rangle^\perp$ is clear. We show the opposite inclusion. Since $X$ has rational singularities if $D$ is a divisor in $\widetilde{X}$, whose class is orthogonal to the exceptional locus of $\pi$, then $D$ restricts trivially on each effective curve supported along this locus. Moreover,
by Artin's contractibility criterion $D\sim D'$, where the support of $D'$ is disjoint from the exceptional locus. Thus $D$ is the pullback of a Cartier divisor on $X$.

Finally, we show (iii). The exact sequence
$$
 \xymatrix{
  0\ar[r] &
  \langle E_1,\dots,E_n\rangle\ar[r] &
  \pi_*^{-1}\Pic(X)\ar[r]^-{\pi_*} &
  \Pic(X)\ar[r]\ar@/^1pc/[l]^-{\pi^*} &
  0
 }
$$
admits a splitting given by the pullback. Thus the central term is the direct sum
of the other two and the statement follows.
\end{proof}

\begin{Remark}
The map $\Pic(\widetilde{X})\to \langle E_1,\dots,E_n\rangle^*$, given by $D\mapsto \langle D,-\rangle$ is not necessarily surjective. For instance, if $\widetilde{X}$ has Picard group
of rank two with intersection matrix
$$
 \begin{pmatrix*}[r]
  4&0\\
  0&-2
 \end{pmatrix*},
$$
then $\widetilde{X}$ is the resolution of a singular quartic surface $X$ of 
$\mathbb P^3$ whose singular locus consists of an ordinary double point. In this case 
$\Pic(X) = \langle H\rangle$, where $H^2=4$ and the exceptional locus consists of just one $(-2)$-curve $E$ coming from the resolution of the singularity. Thus 
$$\pi^*\Pic(X)\oplus \langle E\rangle = \Pic(\widetilde{X})$$ 
since the two lattices have the same rank and intersection matrix. It follows 
that $\Cl(X) = \Pic(X)$.

On the other hand $\langle E\rangle^*$ contains a linear form whose value on $E$ is $1$. Such a linear form can not come from $\Pic(\widetilde{X})$ since the intersection product with $E$ of any element of this lattice is an even number. 
\end{Remark}

\begin{Definition}
A sublattice $L\subseteq\Lambda_{K3}$ is a Mori dream lattice if it belongs to the list
given in~\cite[Theorem 5.1.5.3]{adhl}.
\end{Definition}

The following result generalizes~\cite[Proposition 2.10]{gar}. Indeed we have no assumption on how the $(-1)$-curves of the minimal resolution intersect the exceptional locus.

\begin{theorem}\label{mdsk3}
A singular $K3$ surface whose Picard lattice is Mori dream is a Mori dream surface.
\end{theorem}
\begin{proof}
Let $\Delta\subset\mathbb{C}$ be a disc around the origin. Choose a $1$-parameter family
of period $\omega_t\in T_X\otimes\mathbb C$, whit $t\in\Delta$, such
that 
$$\omega_t^\perp\cap\Lambda_{K3} = L_X,$$ 
for a very general $t\in\Delta$, and 
$$\omega_0^\perp\cap\Lambda_{K3} = L_{\widetilde{X}}.$$ 
By the surjectivity of the global Torelli map this corresponds to a family of $K3$ surfaces over $\Delta$ whose very general member has Picard lattice isometric to $L_X$ and whose 
special fiber over $0\in\Delta$ has Picard lattice isometric to $L_{\widetilde{X}}$.

The choice of an ample class $H\in L_X$ produces a polarized family $\mathcal{X}\rightarrow\Delta$ whose very general member is a smooth projective $K3$ surface $\mathcal X_t$ with Picard lattice isometric to $L_X$ and whose special fiber is isomorphic to the singular $K3$ surface $X$. By hypothesis the very general member is a Mori Dream surface, so that its effective cone is polyhedral. 

By~\cite[Proposition 1.3]{lu} it follows that the effective cone of the central fiber $\mathcal X_0 \simeq X$ is polyhedral as well. On the other hand, if $D$ is a 
nef divisor class on $X$, then $D$ is semiample by the base point free theorem, so $X$ is a Mori dream surface.
\end{proof}

The proof of Theorem~\ref{mdsk3} yields that if a singular $K3$ surface $X$ is a specialization of smooth Mori dream $K3$ surfaces then $X$ is Mori dream. Hence, it allows us to produce singular Mori dream $K3$ surfaces as limits of smooth Mori dream $K3$ surfaces.

On the other hand, not all singular $K3$ surfaces which are Mori dream can be obtained in this way, as shown by the following example.

\begin{Example}\label{SpecF}
Let $\widetilde{X}$ be a $K3$ surface whose Picard
lattice is isometric to the following
$$
\begin{pmatrix*}[r]
  4 & 2 & 0 & 0 \\
  2 & -2 & 1 & 1 \\
  0 & 1 & -2 & 0 \\
  0 & 1 & 0 & -2
\end{pmatrix*}.
$$
Such a surface exists by~\cite[Corollary 3.1]{hu}. Denote by $(H,E_1,E_2,E_3)$ the 
corresponding basis of the Picard lattice. Up to isometry we can assume $H$ to be nef and the remaining three curves to be $(-2)$-curves.

Note that $H$ has even intersection product with any other class in the Picard group. The classes which have intersection product two with $H$ are of the form $D = E_1+a(H-2E_1)+bE_2+cE_3$, with $a,b,c\in\mathbb Z$. The equation $D^2=0$ defines an ellipsoid with no integral points.

Thus we have a birational morphism $\pi\colon {\widetilde{X}}\to X$ and, since $E_2$ and $E_3$ are orthogonal to $H$ and disjoint, it follows that $X$ has two simple double points given by the contraction of these curves. The surface has no other singular points 
since $H^\perp = \langle H-2E_1,E_2,E_3\rangle$ and the equation $(a(H-2E_1)+bE_2+cE_3)^2=-2$ defines an ellipsoid with only two integral points on it.

By Lemma~\ref{res} the Picard lattice of $X$ is $\langle H,-H+2E_1+E_2+E_3\rangle
= \langle E_2,E_3\rangle^\perp$, and it has the following intersection matrix
$$
 \begin{pmatrix*}[r]
  4&0\\
  0&-8
 \end{pmatrix*}.
$$
A smooth element $\mathcal{X}_t$ of the family $\mathcal{X}\to\Delta$ is a smooth $K3$ surface with Picard lattice given by the above $2\times 2$ matrix. By Theorem \ref{thethm} such a surface is not Mori dream since it does not contain classes of self-intersection $0$ or $-2$.

On the other hand, the classes $E_1$ and $H-E_1-E_2-E_3$ are effective since they are of 
self-intersection $-2$ and their intersection with $H$ is $2>0$. Moreover, both are irreducible due to the last intersection value. An easy calculation shows that the 
images of both curves in $X$ have self-intersection $-1$. Thus $X$ is a Mori dream space
by Proposition~\ref{pro:mds}.
\end{Example}

\section{Mori dream $K3$ surfaces of Picard rank two}
In this section we address the case of $K3$ surfaces of Picard rank two an whose singular locus consists of an $A_n$-point.
\begin{Lemma}\label{sa}
Let $X$ be a $\mathbb{Q}$-factorial projective variety and $f: Y\to X$ a resolution of singularities. If every nef class on $Y$ is semiample then the same holds for $X$.
\end{Lemma}
\begin{proof}
Let $D$ be a nef divisor on $X$. Then $f^*D$ is also nef, and by assumption it is semiample. Hence, after 
possibly replacing $f^*D$ with a sufficiently large multiple, we get a divisor $D'$ that is disjoint from the exceptional locus of $f$. Consequently, $D = f_*f^*D$ is semiample.
\end{proof}

\begin{Remark}\label{lu2}
We will need the following observation: the proof of ~\cite[Proposition 1.2]{lu} for the case 
of Picard rank two works verbatim as in the case of higher Picard rank.
\end{Remark}

We will make extensive use of the following result.

\begin{Proposition}\label{pro:mds}
Let $X$ be a, possibly singular, projective $K3$ surface of Picard rank two. Then the following are equivalent:
\begin{itemize}
\item[(i)] $X$ is a Mori dream space.
\item[(ii)] There are two effective divisors $D_1,D_2$ of $X$ such that $D_i^2\leq 0$ and $D_1\cdot D_2>0$.
\end{itemize}
\end{Proposition}
\begin{proof}
We prove (i)$\Rightarrow$(ii). If $X$ is Mori dream then there are two irreducible curves $C_1$, $C_2$ whose classes generate the effective cone of $X$. Thus $C_i^2\leq 0$ for $i=1,2$. By the Hodge index theorem the intersection form on 
$C_1^\perp$ is non-negative definite.

It follows that $C_1^\perp$ intersects the interior of the light cone $Q$, so that the class of $C_2$ lies in
the positive half-plane defined by $C_1$. Thus $C_1\cdot C_2>0$.

We prove (ii)$\Rightarrow$(i). The hypothesis implies that the cone generated by the classes of 
$D_1$ and $D_2$ contains the light cone $Q$. Indeed if $D = uD_1+vD_2$, then 
$$D^2 = u^2D_1^2+2uvD_1\cdot D_2+v^2D_2^2$$ 
can be positive only if $uv>0$ (if both $u$ and $v$ are negative then $D$ can not be effective). As a consequence two irreducible components $C_1$ and $C_2$ in the support of $D_1$ and $D_2$ generate a cone which contains $Q$.

Thus, Remark \ref{lu2} yields that the effective cone of $X$ is polyhedral. Since the resolution of $X$ 
is a smooth $K3$ surface, by Lemma~\ref{sa}, every nef divisor on $X$ is semiample and hence $X$ is Mori dream.
\end{proof}

The following result will help us to find divisors $D_1, D_2$ on a singular $K3$ surface $X$ satisfying (ii) in Proposition \ref{pro:mds}.

\begin{Lemma}\label{FindingAnotherCurve}
Let $X$ be a possibly singular $K3$ surface of Picard rank two. Assume that $X$ does not contain divisors
of self-intersection zero. Let $D_1$ be a divisor on $X$ with negative self-intersection. Then there exists a divisor $D_2$, not linearly equivalent to $-D_1$, such that
$$
D_2^2 = D_1^2 
\quad \text{and} \quad 
D_1 \cdot D_2 > 0.
$$
\end{Lemma}
\begin{proof}
Set $\alpha := D_1^2$, and write the class of $D_1$ with respect to a basis of the divisor class group of $X$. Then, after clearing denominators, the equation $D^2 = \alpha$ takes the form $Q(x,y) = N$, where 
$$
Q(x, y) = a x^2 + b x y + c y^2
$$
is a binary quadratic form and $a, b, c, N\in\mathbb{Z}$. By the Hodge index theorem the discriminant $\Delta := b^2 - 4ac$ is positive. Moreover, the hypothesis that $X$ does not contain divisors with self-intersection zero implies 
that $\Delta$ is not a square. It follows that if the equation $Q(x, y) = N$ has an integral solution then it has infinitely many solutions~\cite[Corollary 11.14]{int}.

On the other hand, given $\alpha, \beta \in \mathbb{Q}$, there are only a finite number of classes $D$ satisfying
$$
D^2 = \alpha \text{ and } D \cdot D_1 = \beta,
$$
since the second equation implies that $D$ lies on a line in the affine plane with coordinates $x, y$, and 
the first equation can have at most two roots on such a line. Hence, there exists a $D_2$ such that 
$D_2^2 = \alpha$, $D_2 \notin \{D_1, -D_1\}$, and $D_1 \cdot D_2 \neq 0$. After possibly replacing $D_2$ with $-D_2$, we can assume $D_1 \cdot D_2 > 0$, proving the statement.
\end{proof}

We now have to address the effectiveness of the divisor $D_2$ whose existence is guaranteed by Lemma \ref{FindingAnotherCurve}. Let $D_1$ and $D_2$ be divisors as in Lemma \ref{FindingAnotherCurve}, in many cases we can prove that $D_2$ is effective provided that $D_1$ is. 

Let $X$ be a singular $K3$ surface and $\pi:\widetilde{X}\to X$ a resolution. Given a Weil divisor $D$ on $X$ we denote by $\bar D$ its strict transform and by $\pi^*D$ its pullback, as defined in~\cite{mum}. Set 
$$
E_D := \pi^*D - \bar{D}
\qquad
\text{ and }
\qquad   
   \{E_D\} := E_D - \lfloor E_D\rfloor.
$$

\begin{Proposition}\label{pro:eff}
With the above notation, if $D^2+\{E_D\}^2 > -4$ then either $D$ or $-D$ is linearly equivalent to an effective divisor.
\end{Proposition}
\begin{proof}
The equality $\pi^*D - \{E_D\} = \bar D + \lfloor E_D\rfloor$ implies that
$$D^2+\{E_D\}^2 = \pi^*D^2+\{E_D\}^2 = (\pi^*D - \{E_D\})^2$$ 
is an even integer. Thus this number must be at least $-2$. On the other hand, it equals the self-intersection of the divisor $D' := \bar D + \lfloor E_D\rfloor$, and thus either $D'$ or $-D'$ is linearly equivalent to an effective divisor. Since $\pi_*D' = D$ the statement follows.
\end{proof}

From the definition of $\{E_D\}$ it follows that this $\mathbb Q$-divisor is effective with rational coefficients smaller than one. Moreover, for any irreducible curve $C$ in the exceptional locus $\Exc(\pi)$ of $\pi$, we have 
$$\{E_D\}\cdot C \in \mathbb Z.$$

This implies that $\{E_D\}$ belongs to the dual lattice $\Hom(\Exc(\pi),\mathbb Z)$. When the singularity is of type
$A_n$ the discriminant group is known to be cyclic
of order $n+1$:
$$
\Hom(\Exc(\pi),\mathbb Z)/\Exc(\pi)\simeq \left\langle \frac{1}{n+1}\sum_{i=1}^niE_i\right\rangle.
$$
Thus $\{E_D\}$ is either the generator or the fractional part of a multiple of this.

\begin{Proposition}\label{pro:mod}
With the above notation, assume that the singular locus of $X$ consists of one point of type $A_n$, with $n\leq 7$, and let 
$E := \frac{1}{n+1}\sum_{i=1}^niE_i$.
Then
$$
\{kE\}^2 = k^2\{E\}^2\pmod 2-2.
$$
\end{Proposition}
\begin{proof}
This can be directly verified by calculating the Cartan matrix of $A_n$ for $n\leq 7$.
\end{proof}

To continue with our analysis we need some preliminary results on Cartan matrices of $A_n$ singularities. 

\subsection*{Cartan matrices of $A_n$ singularities}
Let $C^n$ be the Cartan matrix of the $A_n$ singularity. This is the $n\times n$ matrix whose diagonal entries are $2$, the entries on the superdiagonal and on the subdiagonal are $-1$, and the remaining entries are $0$:
$$
C^n = \left(
\begin{array}{ccccc}
2 & -1 & 0 & \dots & 0 \\ 
-1 & 2 & -1 & \dots & 0 \\ 
\vdots & \vdots & \vdots & \ddots & \vdots \\ 
0 & \dots & -1 & 2 & -1 \\ 
0 & \dots & 0 & -1 & 2
\end{array}
\right).
$$
We will denote by $C^{n}_{i,j}$ the submatrix of $C^n$ obtained by removing the row indexed by $i$ and the column indexed by $j$.

\begin{Lemma}\label{detCartan}
The determinant of $C^n$ is given by $\det(C^n) = n+1$.
\end{Lemma}
\begin{proof}
We have $C^1 = (2)$ and 
$$
C^2 = 
\left(\begin{array}{cc}
2 & -1 \\ 
-1 & 2
\end{array}\right).
$$
Hence, the claim is clear for $n = 1,2$. We proceed by induction on $n$. Expand the determinant with respect to the first row of $C^n$:
$$
\det(C^n) = 2\det(C^{n}_{1,1})+\det(C^{n}_{1,2}).
$$
Note that $C^{n}_{1,1} = C^{n-1}$ and 
$$
C^n_{1,2} = 
\left(\begin{array}{ccccccccc}
-1 & -1 & 0 & \dots & \dots & \dots & \dots & \dots & 0 \\ 
0 & 2 & -1 & 0 & \dots & \dots & \dots & \dots & 0\\
0 & -1 & 2 & -1 & 0 & \dots & \dots & \dots & 0\\
\vdots & \vdots & \vdots & \vdots & \vdots & \ddots & \vdots & \vdots & \vdots\\
0 & 0 & 0 & 0 & 0 & \dots & 0 & -1 & 2
\end{array}\right).
$$
Therefore, $\det(C^n_{1,2}) = -\det(C^{n-2}) = -n+1$ and
$$
\det(C^n) = 2\det(C^{n-1})-\det(C^{n-2}) = 2((n-1)+1) + (-n+1) = n+1 
$$
concluding the proof.
\end{proof}

\begin{Lemma}\label{DiagC}
Let $(C^{n})^{-1}$ be the inverse of $C^n$, and $d_{i,i}$ the $i$-th diagonal entry of $(C^{n})^{-1}$. Then 
$$
d_{i,i} = \frac{i(n-i+1)}{n+1}
$$
for $i = 1,\dots, n$. In particular, $$d_{i,i} = d_{n-i+1,n-i+1}$$ 
for $i = 1,\dots, \lfloor\frac{n+1}{2}\rfloor$.
\end{Lemma}
\begin{proof}
For $n = 1,2$ the claim is clear. We proceed by induction on $n$. Note that 
$$
C^{n}_{i,i} = \left(\begin{array}{cc}
C^{i-1} & 0 \\ 
0 & C^{n-i}
\end{array}\right), 
$$
where the zeros stand for zero matrices of suitable size, for $i = 1,\dots, n$. Therefore, 
$$
d_{i,i} = \frac{(-1)^{2i}\det(C^n_{i,i})}{\det(C^n)} = \frac{\det(C^n_{i,i})}{\det(C^n)}.
$$
By induction we have 
$$
\det(C^n_{i,i}) = \det(C^{i-1})\det(C^{n-i}) = ((i-1)+1)(n-i+1) = i(n-i+1).
$$
Finally, we conclude by Lemma \ref{detCartan}.
\end{proof}

\begin{Proposition}\label{pro:ine}
Let $(C^{n})^{-1}$ be the inverse of $C^n$, and $d_{i,i}$ the $i$-th diagonal entry of $(C^{n})^{-1}$. If $n\geq 8$ then $d_{i,i}\geq 2$ for $i = 3,\dots,n-2$.
\end{Proposition}
\begin{proof}
Since $n\geq 8$ Lemma \ref{DiagC} yields
$$
d_{3,3} = \frac{3(n-2)}{n+1}\geq 2.
$$
The graph of the function $f(x) = \frac{x(n-x+1)}{n+1}$ is an inverted parabola with vertex in $x = \frac{n+1}{2}$. The vertex is therefore the global maximum of $f(x)$. Hence, 
$$
d_{i,i} \geq d_{3,3}\geq 2
$$
for $i = 3,\dots, \lfloor\frac{n+1}{2}\rfloor$. Finally, by Lemma \ref{DiagC} $d_{i,i} = d_{n-i+1,n-i+1}$ for $i = 1,\dots, \lfloor\frac{n+1}{2}\rfloor$ and hence $d_{i,i}\geq 2$ for $i = 3,\dots,n-2$.   
\end{proof}

We are now ready to proceed with the study of the $A_n$ case. 

\begin{Proposition}\label{pro:An}
Assume that the singular locus of $X$ consists of a single point of type $A_n$, with $n\geq 8$. Let $\pi\colon \widetilde X\to X$ be a minimal resolution of singularities with exceptional locus supported on $E_1,\dots,E_n$, and let $D$ be an irreducible curve in $X$ with $D^2<0$. 

Then $\bar D\cdot E_i\neq 0$ for at most one value of $i\in \{1,2,n-1,n\}$, and in this case we have
$$
 D^2
 =
 \begin{cases}
 -\frac{n+2}{n+1} & \text{ if } i\in\{1,n\};\\
 -\frac{4}{n+1}& \text{ if} i\in\{2,n-1\}.
 \end{cases}
$$
\end{Proposition}
\begin{proof}
Recall that 
$$D^2 = \pi^*D^2 = \pi^*D\cdot (\bar D + E_D) = \pi^*D\cdot \bar D = (\bar D + E_D)\cdot \bar D$$
since $\pi^*D$ is orthogonal to $E_D$. Moreover, since $D$ is irreducible, $\bar D$ must
be a $(-2)$-curve. On the other hand 
$$0 = (\bar D+E_D)\cdot E_D$$
yields $E_D\cdot \bar D = -E_D^2$, so that
$$
D^2 = -2 - E_D^2.
$$
Let $E_D = \sum_{k=1}^n\alpha_kE_k$. Then 
$$\beta_i := D\cdot E_i = - E_D\cdot E_i = -(\sum_{k=1}^n\alpha_kE_k)\cdot E_i.$$ 
Set $\alpha := (\alpha_1,\dots,\alpha_n)$, $\beta := (\beta_1,\dots,\beta_n)$, and let $M := (E_i\cdot E_j)$ be the intersection matrix of $E_1,\dots,E_n$. Then $\beta = -\alpha\cdot M$ or equivalently $\alpha = -\beta\cdot M^{-1}$. Thus
$$
D^2 = -2-E_D^2 = -2-\alpha\cdot M\cdot \alpha^t = -2-\beta\cdot M^{-1}\cdot\beta^t.
$$
By definition the vector $\beta$
has non-negative integral entries. If it contains only one non-zero entry, in position $i$, then 
$$-\beta\cdot M^{-1}\cdot\beta^t = -M_{i,i}^{-1},$$
the opposite of the $i$-th diagonal entry of the negative
definite matrix $M^{-1}$. When $n\geq 8$ this entry is at least $2$ as soon as $3\leq i\leq n-2$ by
Proposition~\ref{pro:ine}. If $\beta = (1,0,\dots,0)$ or $\beta = (0,\dots,0,1)$ then, by Lemma~\ref{DiagC}, we have
$$
D^2 = -2-M_{1,1}^{-1} = -2 + \frac{n}{n+1} = -\frac{n+2}{n+1}.
$$
If either $\beta = (0,1,0\dots,0)$ or $\beta = (0,\dots,0,1,0)$ then, by Lemma~\ref{DiagC}, we have
$$
D^2 = -2-M_{2,2}^{-1}  = -2 + \frac{2n-2}{n+1} = -\frac{4}{n+1}
$$
proving the statement.
\end{proof}

We now prove the main result of this section.

\begin{thm}\label{MainAn}
Let $X$ be a $K3$ surface of Picard rank two and whose singular locus consists of a point of type $A_n$. Assume that $X$ contains an irreducible curve of negative self-intersection. If $n\notin\{11,14,15\}$ then $X$ is Mori dream. 
\end{thm}
\begin{proof}
Let $D_1\subset X$ be an irreducible curve of negative self-intersection. By Lemma \ref{FindingAnotherCurve} there is a divisor $D_2$ on $X$ with $D_2^2 < 0$.

As a consequence of Proposition~\ref{pro:mod}, if $n\leq 7$ then $-2\leq \{E_D\}^2< 0$, so that the hypothesis of Proposition~\ref{pro:eff} is fulfilled whenever $D^2\geq -2$.

If $n\geq 8$ and $D$ has the same self-intersection of an irreducible 
negative curve then, by Proposition \ref{pro:An}
$$
 D^2\in\left\{-\frac{n+2}{n+1},-\frac{4}{n+1}\right\}.
$$
On the other hand, $D^2+\{E_D\}^2\equiv 0 \pmod 2$. In general this equation uniquely determines $\{E_D\}^2$, but this is not always the case. When $1\leq n\leq 18$, there are the following exceptions:
\begin{longtable}{c|l|l}
$11$ & $\{E\}^2 = -\frac{11}{12}$ & $\{5E\}^2 = -\frac{35}{12}$\\ 
\midrule
$14$ & $\{E\}^2 = -\frac{14}{15}$ & $\{4E\}^2 = -\frac{44}{15}$\\ 
\midrule
$14$ & $\{2E\}^2 = -\frac{26}{15}$ & $\{7E\}^2 = -\frac{56}{15}$\\ 
\midrule
$15$ & $\{2E\}^2 = -\frac{7}{4}$ & $\{6E\}^2 = -\frac{15}{4}$
\end{longtable}
where $E := \frac{1}{n+1}\sum_{i=1}^niE_i$ is the generator of the discriminant group. In each of these cases 
$D^2+\{E_D\}^2$ equals $-2$ for the first choice of $\{E_D\}$ and $-4$ for the second choice of $\{E_D\}$.
\end{proof}

\section{Codimension one $K3$ surfaces with one $A_n$ singularity}\label{c1s1}

Let $\mathbb{P} := \mathbb{P}(a,b,c,d)$ be a weighted projective
space, and $X\subseteq\mathbb {P}$ an anticanonical surface of degree $a+b+c+d$, whose (possible) singularities are of type $A_n$. Then $X$ is a $K3$ surface which is known to belong to a finite list of $95$ families
\cite{re}.

If $X$ is a very general element in its linear system on $\mathbb P$, then the pullback $\Cl(\mathbb P)
\to \Cl(X)$ is surjective \cite{RS09}, so that the latter group is free of rank one. If $X$ is no longer very 
general then $\Cl(X)$ can have higher rank. 

We are interested in those families where the Picard rank is two and moreover $X$ is a Mori dream surface. The second condition, given the first, is equivalent to the existence of two irreducible curves $\Gamma$, $\Gamma'$ 
on $X$ of negative self-intersection and such that $\Gamma\cdot\Gamma'>0$. 

Our strategy is the following: if one of these curves exists, then one of its irreducible components must have negative self-intersection, we will keep denoting it with $\Gamma$. 

If $\pi:\widetilde X\to X$ is a minimal resolution of singularities, then $\widetilde X$ is a $K3$ surface and the strict transform $\widetilde \Gamma$ is either a $(-2)$-curve or $\widetilde\Gamma^2 = 0$. 

In the second case $\widetilde \Gamma$ does not intersect the exceptional locus of $\pi$, while in the first case it could intersect the exceptional locus but this intersection is bounded by the requirement that $\Gamma^2\leq 0$.

There are two main problems with this approach:
\begin{itemize}
\item[(i)] guarantee the existence of such an $\widetilde X$;
\item[(ii)] prove the existence of the curve $\Gamma'$.
\end{itemize}
For (i) we apply the global Torelli theorem which guarantees that, if $\Lambda$ is an hyperbolic even lattice which embeds into the $K3$ lattice
$$
\Lambda_{K3} := E_8^2\oplus U^3
$$
then there exists a family of $K3$ surfaces whose Picard lattice contains $\Lambda$ and the Picard lattice of the very general element of the family is exactly $\Lambda$. If $\rk(\Lambda)\leq 10$ it is known that $\Lambda$ always embeds into $\Lambda_{K3}$.

About (ii), the existence of $\Gamma'$ is guaranteed by Lemma~\ref{FindingAnotherCurve} and Proposition~\ref{pro:eff} whenever $X$ has only one singularity of type $A_n$ with $n\leq 10$. According to the classification of the $95$
families there are exactly $5$ families where $X$ has only one singular point. 

In each case we provide the value of $H^2$ and the possible values for $\Gamma^2$. We list the possible cases in the following table, where the notation $\{a,b\}$ means that we are free to choose one of the two parameters $a,b$:

\renewcommand{\arraystretch}{1.5}
\begin{longtable}{l|c|c|l|c}\label{thetable}
$X$ & $A_n$ & $\mathbb{P}(d_1,\dots,d_4)$ & Intersection matrix of $H,\Gamma$ & 
${\rm Pic}(\widetilde X)$\\
\hline
$X_{21}$ & $A_9$ & $\mathbb{P}(1,3,7,10)$ & 
{\footnotesize
$\begin{pmatrix}
\frac{1}{10}&n\\n&\{-2,0\}
\end{pmatrix}$ 
$\begin{pmatrix}
\frac{1}{10}&n+\frac{\{3,7\}}{10}\\n+\frac{\{3,7\}}{10}&-\frac{11}{10}
\end{pmatrix}$ 
$\begin{pmatrix}
\frac{1}{10}&n+\frac{\{3,7\}}{5}\\n+\frac{\{3,7\}}{5}&-\frac{2}{5}
\end{pmatrix}$ 
}
&
$E_8\oplus U$
\\
$X_{9}$ & $A_3$ & $\mathbb{P}(1,1,3,4)$ & 
{\footnotesize
$\begin{pmatrix}
\frac{3}{4}&n\\n&\{-2,0\}
\end{pmatrix}$ 
$\begin{pmatrix}
\frac{3}{4}&n+\frac{\{3,5\}}{4}\\n+\frac{\{3,5\}}{4}&-\frac{5}{4}
\end{pmatrix}$ 
$\begin{pmatrix}
\frac{3}{4}&n+\frac{3}{2}\\n+\frac{3}{2}&-1
\end{pmatrix}$ 
}
&
$A_2\oplus U$\\
$X_{10}$ & $A_2$ & $\mathbb{P}(1,1,3,5)$ & 
{\footnotesize
$\begin{pmatrix}
\frac{2}{3}&n\\n&\{-2,0\}
\end{pmatrix}$ 
$\begin{pmatrix}
\frac{2}{3}&n+\frac{\{2,4\}}{3}\\
n+\frac{\{2,4\}}{3}&-\frac{4}{3}
\end{pmatrix}$ 
}
&
$A_1\oplus U$\\
$X_{12}$ & $A_1$ & $\mathbb{P}(1,1,4,6)$ & {\footnotesize
$\begin{pmatrix}
\frac{1}{2}&n\\n&\{-2,0\}
\end{pmatrix}$ 
$\begin{pmatrix}
\frac{1}{2}&n+\frac{1}{2}\\
n+\frac{1}{2}&-\frac{3}{2}
\end{pmatrix}$ 
}
&
$U$\\
$X_{5}$ & $A_1$ & $\mathbb{P}(1,1,1,2)$ & {\footnotesize
$\begin{pmatrix}
\frac{5}{2}&n\\n&\{-2,0\}
\end{pmatrix}$ 
$\begin{pmatrix}
\frac{5}{2}&n+\frac{1}{2}\\
n+\frac{1}{2}&-\frac{3}{2}
\end{pmatrix}$ 
}
 &
$\begin{pmatrix}
2&1\\1&-2
\end{pmatrix}$ 
\\
$X_{6}$ &  & $\mathbb{P}(1,1,1,3)$ & $\begin{pmatrix}
2&n\\n&\{-2,0\}
\end{pmatrix}$ 
 &
$(2)$\\
$X_{4}$ &  & $\mathbb{P}(1,1,1,1)$ & 
$\begin{pmatrix}
4&n\\n&\{-2,0\}
\end{pmatrix}$ 
&
$(4)$\\
\end{longtable}

We discuss the first case. The analysis of the remaining ones being similar. The minimal resolution $\pi:
\widetilde X\to X$ is a $K3$ surface whose Picard group, according to~\cite{bel}, is the lattice $E_8\oplus U$, we denote by $e_1,\dots,e_8,f_1,f_2$ the generators. Then we can assume one of the $f_i$ to be nef, let us say $f_1$.

Thus $f_1$ defines an elliptic fibration which contains the whole $E_8$ diagram of curves in a fiber and to complete the fiber one has to add the curve whose class is 
$$e_9 := f_1-(3e_1-2e_2-4e_3-6e_4-5e_5-4e_6-3e_7-2e_8)$$ 
according to the following graph:
\begin{center}
\begin{tikzpicture}[scale=1.5]
  \node (e8) at (0,0) {\(e_8\)};
  \node (e7) at (1,0) {\(e_7\)};
  \node (e6) at (2,0) {\(e_6\)};
  \node (e5) at (3,0) {\(e_5\)};
  \node (e4) at (4,0) {\(e_4\)};
  \node (e3) at (5,0) {\(e_3\)};
  \node (e2) at (6,0) {\(e_2\)};
  \node (e1) at (4,-1) {\(e_1\)};
  \draw (e8) -- (e7) -- (e6) -- (e5) -- (e4) -- (e3) -- (e2);
  \draw (e4) -- (e1);
\end{tikzpicture}
\end{center}
A section of this elliptic fibration has class $f_2-f_1$ and it intersects only $e_9$ in a point. The $A_9$ curve is thus given by $e_2,\dots,e_9,f_2-f_1$. After contracting these curves we are left with the class $h$, given by the image of $e_1$. 

Moving $\widetilde X$ in this family we obtain surfaces of higher Picard rank, in particular we are interested in those of Picard rank $11$ which contain a new curve $\widetilde\Gamma$ of self-intersection $-2$ or $0$. Set
$$
n := \widetilde\Gamma\cdot e_1.
$$ 
If $\widetilde\Gamma$ is disjoint from the $A_9$, after contracting it, we get the first type of intersection matrix. The second type is obtained assuming that $\widetilde\Gamma$ is a $(-2)$-curve which intersects only one of the extremal curves of the $A_9$, namely $e_2$ or $f_2-f_1$. Finally, the third type is obtained assuming that $\widetilde\Gamma$ is a $(-2)$-curve which intersects either $e_3$ or $e_8$.

In what follows we will provide an explicit example in each case and describe its geometry. Let $\Gamma \subseteq \mathbb{P}$ be an irreducible curve, and $\mathcal{L}$ the linear system of, possibly singular, $K3$ surfaces in $\mathbb{P}$ of degree $a+b+c+d$ that contain $\Gamma$.  

Let $\widetilde{X} \to X$ denote the resolution of $X \in \mathcal{L}$ with exceptional locus $E$, and $\Lambda_E$ be the sublattice of the Picard group generated by the classes of the irreducible components of $E$. 

\begin{Proposition}\label{PAut}
Let $X$ be a very general element of $\mathcal{L}$, and $\Aut_\Gamma(\mathbb{P}) \subseteq \Aut(\mathbb{P})$ the subgroup of automorphisms of $\mathbb{P}$ that preserve the curve $\Gamma$. If the following equality 
$$
\dim (\mathcal{L}) - \dim \Aut_\Gamma(\mathbb{P}) = 18 - \rk(\Lambda_E)
$$
holds, then $\rk(\Pic(X)) = 2$.
\end{Proposition}
\begin{proof}
By definition, the Picard lattice $\Pic(\widetilde{X})$ contains the sublattice $\Lambda_E \oplus \langle H_{|X}, 
\widetilde{\Gamma} \rangle$, where $H$ is the ample generator of the Picard group of $\mathbb{P}$, and $\widetilde{\Gamma}\subseteq\widetilde X$ is the strict transform of $\Gamma$.

For every $X \in \mathcal{L}$, the set of surfaces $Y \in \mathcal{L}$ that are isomorphic to $X$ under $\Aut_\Gamma(\mathbb{P})$ lies in at most a countable number of orbits. This follows from the observation that any isomorphism $\phi : Y \to X$ defines an embedding of $X$ in $\mathbb{P}$ via the linear system $|\phi_*(H_{|Y})|$, which maps $\Gamma$ to itself. 

Since the Picard group of $X$ is countable, there are at most countably many such embeddings up to the action of ${\rm Aut}_\Gamma(\mathbb{P})$. As a consequence, the dimension of the moduli space of, possibly singular, $K3$ surfaces in $\mathcal{L}$ is
$$
\dim \mathcal{L} - \dim {\rm Aut}_\Gamma(\mathbb{P}).
$$
By the global Torelli theorem, the dimension of the moduli space of smooth $K3$ surfaces with Picard lattice of rank 
$r$ is $20 - r$. Therefore
$$
\dim (\mathcal{L})- \dim \Aut_\Gamma(\mathbb{P}) = 20 - \rk(\Pic(\widetilde{X}))\leq 18 - \rk(\Lambda_E).
$$
If the equality holds, we have 
$$\rk(\Pic(\widetilde{X})) = {\rm rk}(\Lambda_E) + 2$$ 
or equivalently $\rk(\Pic(X)) = 2$ as claimed.
\end{proof}

Next we will show that Proposition \ref{PAut} applies to the surfaces $X_{21},X_{12},X_{10},X_{9},X_{5}$. We list the relevant numbers in the following table:
$$
\begin{array}{c|c|c|c}
X & \dim \mathcal{L} & \dim {\rm Aut}_\Gamma(X) & {\rm rk} \, \Lambda_E \\ 
\hline
X_{21} \, \text{in} \, \mathbb{P}(1,3,7,10) & 10 & 1 & 9 \\ \hline
X_{12} \, \text{in} \, \mathbb{P}(1,1,4,6) & 25 & 8 & 1\\ \hline
X_{10} \, \text{in} \, \mathbb{P}(1,1,3,5) & 24 & 8 & 2\\ \hline
X_{9} \, \text{in} \, \mathbb{P}(1,1,3,4) & 22 & 7 & 3\\ \hline
X_{5} \, \text{in} \, \mathbb{P}(1,1,1,2) & 27 & 10 & 1
\end{array}
$$

For each of these type of surfaces we will compute explicit examples, we will extensively use the following simple result.

\begin{Lemma}\label{intwps}
Let $H_{d_1},\dots,H_{d_n}\subset\mathbb{P}(a_0,\dots,a_n)$ be hypersurfaces of degree $d_1,\dots,d_n$ respectively. Then the intersection number of the $H_{d_i}$ is given by 
$$
H_{d_1}\cdots H_{d_n} = \frac{\prod_{i=1}^{n}d_i}{\prod_{i=0}^{n}a_i}.
$$
\end{Lemma}
\begin{proof}
The projection
$$
\begin{array}{cccc}
\pi: & \mathbb{P}^n & \longrightarrow & \mathbb{P}(a_0,\dots,a_n) \\
     & [x_0: \dots :x_n] & \longmapsto & [x_0^{a_0}: \dots :x_n^{a_n}]
\end{array}
$$
has degree 
$$
\deg(\pi) = \prod_{i=0}^{n}a_i.
$$
On one hand, $\pi^{*}(H_{d_1}\cdots H_{d_n}) = \pi^{*}H_{d_1}\cdots \pi^{*}H_{d_n} = \prod_{i=1}^{n}d_i$ since $\pi^{*}H_{d_i}$ is a hypersurface of degree $d_i$ in $\mathbb{P}^n$, and on the other hand $\pi^{*}(H_{d_1}\cdots H_{d_n}) = \deg(\pi)(H_{d_1}\cdots H_{d_n})$. Hence  
$$
(H_{d_1}\cdots H_{d_n})\cdot\prod_{i=0}^{n}a_i = (H_{d_1}\cdots H_{d_n})\cdot\deg(\pi) = \pi^{*}(H_{d_1}\cdots H_{d_n}) = \prod_{i=1}^{n}d_i
$$
and the claim follows.
\end{proof}

As we said all intersection numbers in the following examples are computed using Lemma \ref{intwps}.

\subsubsection{$X_{21}\subset\mathbb{P}(1,3,7,10)$}
Consider the map
$$
\begin{array}{cccc}
 \overline{\gamma}: & \mathbb{P}^1 & \longrightarrow & \mathbb{P}(1,3,7,10)\\
  & [u:v] & \mapsto & [u^3:u^4v^5:uv^{20}:v^{30}]
\end{array}
$$
and let $\Gamma_1$ be its image. Let $G$ be the group of automorphisms of $\mathbb{P}(1,3,7,10)$ stabilizing $\Gamma_1$. An automorphism $\alpha\in G$ fixes $[0:0:0:1]$, and therefore yields an automorphism $\overline{\alpha}$ of $\mathbb{P}^1$ fixing $[0:1]$, that is of the form  
$$
\begin{array}{cccc}
 \overline{\alpha}: & \mathbb{P}^1 & \longrightarrow & \mathbb{P}^1\\
  & [u:v] & \mapsto & [au:bu+v].
\end{array}
$$
Now, via $\overline{\gamma}$, such automorphism induces the following transformation
$$x = a^3u^3, y = a^4u^4(bu+v)^5, z = au(bu+v)^{20}, w = (bu+v)^{30}.$$
If $b = 0$ this transformation is the restriction to $\Gamma_1$ of the following automorphism of $\mathbb{P}(1,3,7,10)$:
$$
\begin{array}{cccc}
\alpha: & \mathbb{P}(1,3,7,10) & \longrightarrow & \mathbb{P}(1,3,7,10)\\
  & [x:y:z:w] & \mapsto & [a^3x: a^4y: az: w].
\end{array}
$$
On the other hand, if $b\neq 0$ there is no way to obtain for instance the monomial $u^4u^4v = u^8v$ as a product of $x = a^3u^3, y = a^4u^4(bu+v)^5, z = au(bu+v)^{20}, w = (bu+v)^{30}$. Therefore, $\overline{\alpha}$ does not come from an automorphism of $\mathbb{P}(1,3,7,10)$, and hence $G\cong \mathbb{C}^{*}$.

Consider the surface $X_{21} = \{f = 0\}$ where
$$
\begin{array}{ll}
f = & x^{14}z + x^{11}yz + 2x^{11}w - x^9y^4 - x^8y^2z + 2x^8yw + x^7z^2 -x^6y^5 - x^5y^3z +  \\
 & x^4yz^2 + x^4zw - x^3y^6 - x^2y^4z - x^2y^3w - xy^2z^2 + xw^2 - y^7 - z^3.
\end{array}
$$
Then $X_{21}$ is a $K3$ surface whose singular locus consists of a $\frac{1}{10}(3,7)$ singularity at $[0:0:0:1]$.

Now, let $S = \{x^5z - y^4 = 0\}\subset\mathbb{P}(1,3,7,10)$. The intersection $X_{21}\cap S$ has two irreducible components
$$
X_{21}\cap S = \Gamma_1 \cap \Gamma_2
$$
where $\Gamma_2$ is the curve whose ideal is generated by the following four polynomials
$$
\begin{array}{l}
x^{11}w - x^8y^2z + x^8yw - x^5y^3z + \frac{1}{2}(x^4yz^2 + x^4zw -
        x^2y^3w - xy^2z^2 + xw^2 - z^3);\\
x^{10}y^2 + x^7y^3 - \frac{1}{2}(x^6yz - x^3y^2z - x^2z^2 - y^2w);\\
x^{13} + x^{10}y + \frac{1}{2} (x^6z - x^4y^3 + x^3w + y^2z);\\
x^5z - y^4.
\end{array} 
$$
Consider the projection $\pi:\mathbb{A}^4\setminus\{(0,0,0,0)\}\rightarrow \mathbb{P}(1,3,7,10)$. The intersection $\pi^{-1}(\Gamma_1)\cap \pi^{-1}(\Gamma_2)$ has a reduced component that gets mapped by $\pi$ to six simple points and a component of multiplicity $16$ that gets mapped by $\pi$ to $[0:0:0:1]$. Hence, $\Gamma_1\cdot\Gamma_2 = 6+\frac{6}{10} = \frac{38}{5}$. Moreover, $\Gamma_1\cdot H = \frac{6}{10}$. Since $\Gamma_1 + \Gamma_2$ is cut out in $X_{21}$ by a surface of degree $12$ we have
$$
12H = \Gamma_1 + \Gamma_2
$$
and intersecting with $\Gamma_1$ we get
$$
\frac{36}{5} = 12 H\cdot\Gamma_1 = \Gamma_1^2 + \Gamma_1\cdot\Gamma_2 = \Gamma_1^2 + \frac{38}{5}.
$$
So, $\Gamma_1^2 = -\frac{2}{5}$. Now, from 
$$
\frac{72}{5} = 144H^2 = \Gamma_1^2 + 2\Gamma_1\cdot\Gamma_2 + \Gamma_2^2 = -\frac{2}{5} + \frac{76}{5} + \Gamma_2^2
$$
so that $\Gamma_2^2 = -\frac{2}{5}$.

Finally, we found two effective curves such that  
$$
\Gamma_1^2 = \Gamma_2^2 =  -\frac{2}{5},\: \Gamma_1\cdot\Gamma_2 = \frac{38}{5}
$$
and Proposition \ref{pro:mds} yields that a very general member of the linear system of $K3$ surfaces containing $\Gamma_{1}$ is a Mori dream space. 

\subsection{$X_{2+a+b}\subset\mathbb{P}(1,1,a,b)$}
In the following three examples we will always be in the following situation: 
$$X_{2+a+b}\subset\mathbb{P}(1,1,a,b)$$ 
will be an anticanonical surface, hence of degree $2+a+b$, general among those containing a component $\Gamma_1$ of the reducible curve 
$$\Gamma_1\cup\Gamma_2 = \{z = 0\}\cap X_{2+a+b}\subset\mathbb{P}(1,1,a,b),$$ 
where $[x:y:z:w]$ are homogeneous coordinates on $\mathbb{P}(1,1,a,b)$.
Moreover, one of the following conditions will hold:
\begin{itemize}
\item[-] $\gcd(a,b) = 1$ and $a$ divides $2+a+b$;
\item[-] $\gcd(a,b) > 1$ and both $a$ and $b$ divide $2+a+b$.
\end{itemize}
The singular locus of $X_{2+a+b}$ will consist of a single $A_k$ point $p\in X_{2+a+b}$, located at $p := [0:0:0:1]$ in the first case and along the singular ambient curve $x=y=0$ in the second case. 

\subsubsection{$X_{12}\subset\mathbb{P}(1,1,4,6)$}
Consider the $K3$ surface
$$
X_{12} = \{zA_8 + wB_6 = 0\}\subset\mathbb{P}(1,1,4,6)
$$
with $A_{8}\in k[x,y,z,w]_{8}$ and $B_{6}\in k[x,y,z,w]_{6}$. The singular locus of $X_{12}$ consists of a $\frac{1}{2}(1,1)$ singularity at $[0:0:1:0]$. The curve $\Gamma_1 \cup \Gamma_2$ is cut out by $\{z = 0\}$. So
$$
8 = 16H^2 = (\Gamma_1 + \Gamma_2)^2 = \Gamma_1^2 + 2\Gamma_1\cdot\Gamma_2 + \Gamma_2^2.
$$
Since $\Gamma_1 \cap \Gamma_2 = \{z = w = B_{6} = 0\}$ we have $\Gamma_1\cdot \Gamma_2 = 6$. From $4H = \Gamma_1 + \Gamma_2$ we get 
$$
4H\cdot\Gamma_1 = \Gamma_1^2 + \Gamma_1\cdot\Gamma_2
$$
and hence $\Gamma_1^2 = \Gamma_2^2 = -2$. Finally, we found two effective curves such that  
$$
\Gamma_1^2 = \Gamma_2^2 = -2,\: \Gamma_1\cdot\Gamma_2 = 6
$$
and Proposition \ref{pro:mds} yields that a very general member of the linear system of $K3$ surfaces containing $\Gamma_{1}$ is a Mori dream space. 

\subsubsection{$X_{10}\subset\mathbb{P}(1,1,3,5)$}\label{ex:1135}
Consider the $K3$ surface
$$
X_{10} = \{zA_7 + wB_5 = 0\}\subset\mathbb{P}(1,1,3,5)
$$
where $A_7\in k[x,y,z,w]_7, B_5\in k[x,y,z,w]_5$. Note that $X_{10}$ has a unique singular point at $[0:0:1:0]$ which is a singularity of type $\frac{1}{3}(1,2)$. The curve $\Gamma_1 \cup \Gamma_2$ is cut out by $\{z = 0\}$, so
$$
6 = (\Gamma_1 + \Gamma_2)^2 = \Gamma_1^2 + 2\Gamma_1\cdot\Gamma_2 + \Gamma_2^2.
$$
Moreover, $\Gamma_1 \cap \Gamma_2 = \{z = w = B_5 = 0\}$ yields $\Gamma_1\cdot \Gamma_2 = 5$. Set $\Gamma_3 = \{w = A_7 = 0\}$. Then $\Gamma_1 \cup \Gamma_3$ is cut out by $\{w = 0\}$. Hence
$$
\frac{50}{3} = (\Gamma_1 + \Gamma_3)^2 = \Gamma_1^2 + 2\Gamma_1\cdot\Gamma_3 + \Gamma_3^2
$$
and $\Gamma_1 \cap \Gamma_3 = \{z = w = A_7 = 0\}$ yields $\Gamma_1\cdot \Gamma_3 = 7$. We have $5H = \Gamma_1 + \Gamma_3$ and
$5H\cdot\Gamma_3 = \Gamma_1\cdot\Gamma_3 + \Gamma_3^2 = \frac{35}{3}$. So
$$
\Gamma_3^2 = \frac{35}{3} - \Gamma_1\cdot \Gamma_3 = \frac{14}{3}.
$$
Now 
$$\frac{50}{3} = \Gamma_1^2 + 2\Gamma_1\cdot \Gamma_3 + \Gamma_3^2 = \Gamma_1^2 + \frac{56}{3}
$$ 
yields $\Gamma_1^2 =-2$ and from
$$
6 = \Gamma_1^2 + 2\Gamma_1\cdot \Gamma_2 + \Gamma_2^2 = 8 + \Gamma_2^2
$$
we get $\Gamma_2^2 = -2$. Finally, we got two effective curves such that 
$$
\Gamma_1^2 = \Gamma_2^2 = -2,\: \Gamma_1\cdot\Gamma_2 = 5
$$
and Proposition \ref{pro:mds} yields that a very general member of the linear system of $K3$ surfaces containing $\Gamma_{1}$ is a Mori dream space.

\subsubsection{$X_{9}\subset\mathbb{P}(1,1,3,4)$}\label{ex:1134}
Consider the $K3$ surface
$$
X_{9} = \{zA_6 + wB_5 = 0\}\subset\mathbb{P}(1,1,3,4)
$$
where $A_6\in k[x,y,z,w]_6, B_5\in k[x,y,z,w]_5$. We have 
$$
\Gamma_1 = \{z = w = 0\} \text{ and } \Gamma_2 = \{z = B_5 = 0\}.
$$
Note that $X_{9}$ has a unique singular point at $[0:0:0:1]$ which is a singularity of type $\frac{1}{4}(1,3)$. The curve $\Gamma_1 \cup \Gamma_2$ is cut out by $\{z = 0\}$, so
$$
\frac{27}{4} = (\Gamma_1 + \Gamma_2)^2 = \Gamma_1^2 + 2\Gamma_1\cdot\Gamma_2 + \Gamma_2^2.
$$
Moreover, $\Gamma_1 \cap \Gamma_2 = \{z = w = B_5 = 0\}$ yields $\Gamma_1\cdot \Gamma_2 = 5$. Set $\Gamma_3 = \{w = A_6 = 0\}$. Then $\Gamma_1 \cup \Gamma_3$ is cut out by $\{w = 0\}$. Hence
$$
12 = (\Gamma_1 + \Gamma_3)^2 = \Gamma_1^2 + 2\Gamma_1\cdot\Gamma_3 + \Gamma_3^2
$$
and $\Gamma_1 \cap \Gamma_3 = \{z = w = A_6 = 0\}$ yields $\Gamma_1\cdot \Gamma_3 = 6$. We have $4H = \Gamma_1 + \Gamma_3$ and
$4H\cdot\Gamma_3 = \Gamma_1\cdot\Gamma_3 + \Gamma_3^2 = 8$. So
$$
\Gamma_3^2 = 8 - \Gamma_1\cdot \Gamma_3 = 2.
$$
Now 
$$12 = \Gamma_1^2 + 2\Gamma_1\cdot \Gamma_3 + \Gamma_3^2 = \Gamma_1^2 +14
$$ 
yields $\Gamma_1^2 = -2$ and from
$$
\frac{27}{4} = \Gamma_1^2 + 2\Gamma_1\cdot \Gamma_2 + \Gamma_2^2
$$
we get $\Gamma_2^2 = -\frac{5}{4}$. Finally, we got two effective curves such that 
$$
\Gamma_1^2 = -2,\: \Gamma_2^2 = -\frac{5}{4},\: \Gamma_1\cdot\Gamma_2 = 5
$$
and Proposition \ref{pro:mds} yields that a very general member of the linear system of $K3$ surfaces containing $\Gamma_{1}$ is a Mori dream space. 

\subsubsection{$X_{5}\subset\mathbb{P}(1,1,1,2)$}\label{ex:1112}
Consider the $K3$ surface
$$
X_5 = \{yA_4 + wB_3 = 0\}\subset\mathbb{P}(1,1,1,2)
$$
where $A_4\in k[x,y,z,w]_4, B_3\in k[x,y,z,w]_3$. We have 
$$
\Gamma_1 = \{y = w = 0\} \text{ and } \Gamma_2 = \{y = B_3 = 0\}.
$$
Note that $X_5$ has a unique singular point at $[0:0:0:1]$ which is a singularity of type $\frac{1}{2}(1,1)$. Since $\Gamma_1\cdot\Gamma_2 = 3$ We have 
$$1 = H\cdot\Gamma_1 = \Gamma_1^2 + \Gamma_1\cdot\Gamma_2 = \Gamma_1^2 + 3.$$
So $\Gamma_1^2 = -2$. Furthermore, 
$$
\frac{5}{2} = H^2 = \Gamma_1^2 + 2\Gamma_1\cdot\Gamma_2 + \Gamma_2^2 = -2 + 6 + \Gamma_2^2
$$
yields $\Gamma_2^2 = -\frac{3}{2}$. Summing-up we found two effective curves such that 
$$
\Gamma_1^2 = -2,\: \Gamma_2^2 = -\frac{3}{2},\: \Gamma_1\cdot\Gamma_2 = 3
$$
and Proposition \ref{pro:mds} yields that a very general member of the linear system of $K3$ surfaces containing $\Gamma_{1}$ is a Mori dream space. 

\subsection{Magma scripts}
The Magma code used in this paper is available at 
\begin{center}
\url{https://github.com/alaface/singk3}
\end{center}

\bibliographystyle{amsalpha}
\bibliography{Biblio}

\providecommand{\bysame}{\leavevmode\hbox to3em{\hrulefill}\thinspace}
\providecommand{\MR}{\relax\ifhmode\unskip\space\fi MR }
\providecommand{\MRhref}[2]{%
  \href{http://www.ams.org/mathscinet-getitem?mr=#1}{#2}
}
\providecommand{\href}[2]{#2}
\begin{thebibliography}{ADHL15}

\bibitem[ADHL15]{adhl}
Ivan Arzhantsev, Ulrich Derenthal, J\"{u}rgen Hausen, and Antonio Laface,
  \emph{Cox rings}, Cambridge Studies in Advanced Mathematics, vol. 144,
  Cambridge University Press, Cambridge, 2015. \MR{3307753}

\bibitem[AHL10]{ahl}
Michela Artebani, J\"urgen Hausen, and Antonio Laface, \emph{On {C}ox rings of
  {K}3 surfaces}, Compos. Math. \textbf{146} (2010), no.~4, 964--998.
  \MR{2660680}

\bibitem[Bel02]{bel}
Sarah-Marie Belcastro, \emph{Picard lattices of families of {$K3$} surfaces},
  Comm. Algebra \textbf{30} (2002), no.~1, 61--82. \MR{1880661}

\bibitem[Gar17]{gar}
Alice Garbagnati, \emph{Mori dream spaces extremal contractions of {K}3
  surfaces}, Osaka J. Math. \textbf{54} (2017), no.~3, 409--433. \MR{3685585}

\bibitem[HK00]{HuKeel}
Y.~Hu and S.~Keel, \emph{Mori dream spaces and {GIT}}, Michigan Math. J.
  \textbf{48} (2000), 331--348, Dedicated to William Fulton on the occasion of
  his 60th birthday. \MR{1786494}

\bibitem[Huy16]{hu}
Daniel Huybrechts, \emph{Lectures on {K}3 surfaces}, Cambridge Studies in
  Advanced Mathematics, vol. 158, Cambridge University Press, Cambridge, 2016.
  \MR{3586372}

\bibitem[LHV23]{LealLozanoMontserrat}
Manuel Leal, César~Lozano Huerta, and Montserrat Vite, \emph{The
  {N}oether-{L}efschetz locus of surfaces in $\mathbb{P}^3$ formed by
  determinantal surfaces}, https://arxiv.org/abs/2303.09028, 2023.

\bibitem[LU24]{lu}
Antonio Laface and Luca Ugaglia, \emph{Effective cone of the blowup of the
  symmetric product of a curve}, Proc. Amer. Math. Soc. Ser. B \textbf{11}
  (2024), 229--242. \MR{4762685}

\bibitem[Mum61]{mum}
David Mumford, \emph{The topology of normal singularities of an algebraic
  surface and a criterion for simplicity}, Inst. Hautes \'{E}tudes Sci. Publ.
  Math. (1961), no.~9, 5--22. \MR{153682}

\bibitem[Nik79]{Ni}
V.~V. Nikulin, \emph{Integer symmetric bilinear forms and some of their
  geometric applications}, Izv. Akad. Nauk SSSR Ser. Mat. \textbf{43} (1979),
  no.~1, 111--177, 238. \MR{525944}

\bibitem[Rei80]{re}
Miles Reid, \emph{Canonical {$3$}-folds}, Journ\'{e}es de {G}\'{e}ometrie
  {A}lg\'{e}brique d'{A}ngers, {J}uillet 1979/{A}lgebraic {G}eometry, {A}ngers,
  1979, Sijthoff \& Noordhoff, Alphen aan den Rijn---Germantown, Md., 1980,
  pp.~273--310. \MR{605348}

\bibitem[RS09]{RS09}
G.~V. Ravindra and V.~Srinivas, \emph{The {N}oether-{L}efschetz theorem for the
  divisor class group}, J. Algebra \textbf{322} (2009), no.~9, 3373--3391.
  \MR{2567426}

\bibitem[Wei17]{int}
Martin~H. Weissman, \emph{An illustrated theory of numbers}, American
  Mathematical Society, Providence, RI, 2017. \MR{3677120}

\end{thebibliography}

\end{document}